\documentclass[12pt,a4paper,reqno]{amsart}
\usepackage{mathrsfs}
\usepackage{calc}
\usepackage{cite}
\usepackage{amssymb,amsmath,latexsym, amsthm,enumerate,amsfonts,graphicx}
\usepackage[mathscr]{euscript}
\input amssym.def
\input amssym

\usepackage{amssymb}

\def\dj{d\kern-0.4em\char"16\kern-0.1em}
\def\Dj{\mbox{\raise0.3ex\hbox{-}\kern-0.4em D}}

\def\be{\begin{equation}}
\def\ee{\end{equation}}
\def\bena{\begin{eqnarray*}}
\def\ena{\end{eqnarray*}}

\def\t{\tau}
\def\s{\sigma}

\def\suml{\sum\limits}

\def\dss{\displaystyle}

\newcommand{\WF}{\operatorname{WF}}

 \def\D{\mathcal{D}}
 \def\E{\mathcal{E}}
 \def\Rd{\mathbf{R}^d}
 \def\Z{\mathbf{Z}_+}

\def\N{\mathbf{N}}

\def\lf {\lfloor}
\def\rf{\rfloor}

\newcommand{\sing}{\operatorname{singsupp}}
\newcommand{\supp}{\operatorname{supp}}
\numberwithin{equation}{section}

\newtheorem{te}{Theorem}[section]
\newtheorem{lema}{Lemma}[section]
\newtheorem{prop}{Proposition}[section]
\newtheorem{cor}{Corollary}[section]

\theoremstyle{definition}

\newtheorem{de}{Definition}[section]
\theoremstyle{remark}
\newtheorem{rem}{Remark}[section]

\begin{document}

%
%
%
%
%
%
%
%
%

\title{Ultradifferentiable functions of class $M_p^{\t,\s}$ and microlocal regularity}

\author{Nenad Teofanov}

\address{Department of Mathematics and Informatics,
Faculty of Sciences, University of Novi Sad, Novi Sad, Serbia}

\email{nenad.teofanov@dmi.uns.ac.rs}

\author{Filip Tomi\'c}

\address{Faculty of Technical Sciences,
University of Novi Sad, Novi Sad, Serbia}

\email{filip.tomic@uns.ac.rs}





\begin{abstract}
We study spaces of ultradifferentiable functions which contain Gevrey classes.
Although the corresponding defining sequences do not satisfy Komatsu's condition (M.2)',
we prove appropriate continuity properties under the action of (ultra)differentiable operators.
Furthermore, we study convenient localization procedure which leads to the
concept of wave-front set with respect to our regularity conditions. As an application,
we identify the standard projections of intersections/unions of wave-front sets
as singular supports of suitable
spaces of ultradifferentiable functions.
\end{abstract}

\maketitle

\emph{Dedicated to Professor Pilipovi\'c on the occasion of his
$65^{\text{th}}$ birthday.}

\section{Introduction}\label{sec-0}

\par

Gevrey classes were introduced to describe regularity properties of
fundamental solution of the heat operator in \cite{Gevrey}, and thereafter
used in the study of different aspects of general theory
of linear partial differential operators such as hypoellipticity,
local solvability and propagation of singularities. We refer to
\cite{Rodino} for a detailed exposition of Gevrey
classes and their applications to the theory of linear partial
differential operators.
The intersection (projective limit) of Gevrey classes is strictly larger than the space of analytic functions
while its union (inductive limit) is strictly contained in the class of
smooth functions.
Therefore, it is of interest to study the intermediate spaces of smooth functions
which are contained in those gaps by introducing appropriate regularity conditions. On one hand, this may serve to
describe hypoellipticity properties between smooth/analytic hypoellipticity and
Gevrey hypoellipticity. On other hand,
it can be used in the study of corresponding microlocal regularity properties.

\par

In this paper we continue and complement our research initiated by Professor Stevan Pilipovi\'c  and recently published in
\cite{PTT-01, PTT-02},  and show further properties of classes of ultradifferentiable functions
which contain Gevrey classes. Recall, in \cite{PTT-01} we introduced sequences
$M_p^{\t,\s}=p^{\t p^{\s}}$, $p\in \N $, $\t>0$ and  $\s>1$, and
used them to define and study test function spaces for Roumieu type ultradistribution.
That approach is further developed in   \cite{PTT-02} where, together with a more detailed analysis of
ultradifferentiable functions of class $M_p^{\t,\s}$, we perform microlocal analysis with respect to
the regularity of such classes. More precisely, we proved there  that
$$\WF_{0,\infty}(P(D)u)\subseteq \WF_{0,\infty}(u)\subseteq \WF_{0,\infty}(P(D)u) \cup {\rm Char}(P),$$
where $u$  is a Schwartz distribution, $P(D)$ is a  partial differential operator with constant coefficients and
characteristic set ${\rm Char}(P)$,
and $\WF_{0,\infty}$ denotes the wave front set described in terms of new regularity conditions, see Section \ref{sec-3}.

\par

Different types of wave front sets are introduced in connection to the equation under investigation.
For example, the Gabor wave front set from \cite{HermanderRad} and \cite{RodinoWahlberg} is recently successfully
applied to different situations including the study of Schr\"odinger equations,
see \cite{CS, CW, CNR-1, CNR-2, P-SRW,  SW-1, SW-2, Wahlberg} and the references therein.
Such wave-front set can be characterized in terms of rapid decay of its
Gabor coefficients. That idea is introduced and exploited in \cite{JPTT, PTToft-01, PTToft-02}, and extended in \cite{CJT-1, CJT-2}
to more general Banach and Fr\'echet spaces.
The main tool used there are methods of time-frequency analysis and modulation spaces.
We refer to \cite{F1,FG1, FS1, FS2} for
details on modulation spaces and their role in time-frequency analysis, and remark that
a version of Gabor wave front set adapted to  regularity proposed in
this paper, will be the subject of our future investigation.

\par

This Section is ended by fixing the notation and recalling the standard definition of
ultradifferentiable functions and wave-front sets, and the reader familiar with the subject may
proceed to Section \ref{sec-1} which is devoted to the definition and basic properties of regularity classes
$\E_{\t,\s} (U)$. In particular, we study their embeddings with respect to parameters
$\t>0$ and  $\s>1$ (Proposition \ref{detectposition}), and show the stability under differentiation (Theorem \ref{tealgebra}), although its defining sequence
$M_p^{\t,\s}=p^{\t p^{\s}}$ does not satisfy (M.2)' (cf. Subsection \ref{definicije}). Furthermore, we study
the continuity of certain ultradifferentiable operators (Theorem \ref{TeoremKonstant}).

For the purpose of local analysis in Section \ref{sec-2}
we introduce particular admissibility condition for sequences of cut-off functions see Definition \ref{definicijaNiza},
and discuss regularity of Schwartz distributions in Propositions \ref{dovoljanUslov} and \ref{potrebanUslov}.

In Section \ref{sec-3} we first recall the definition of the wave-front set from   \cite{PTT-02}
and prove Lemma \ref{Singsuplema}, an important auxiliary result which is used
in the proof of  the pseudolocal property in Subsection \ref{pseloc}.
We conclude the paper by identifying the standard projections of intersections
and unions of wave-front sets with singular supports of appropriate projective/inductive limits of
test function spaces, Theorem \ref{projections}.

We note that  Propositions \ref{dovoljanUslov} and \ref{potrebanUslov} and  Lemma \ref{Singsuplema}
are stated in \cite{PTT-02} without proofs.

\subsection{Notation} \label{subsec01}
Sets of numbers are denoted in a usual way, e.g. ${\bf N}$ (resp. $\Z$) denotes the set of nonnegative ( resp. positive) integers.
For $x\in {\bf R}_+$ the floor
function is denoted by $\lf x \rf:=\max\{m\in
\N\,:\,m\leq x\}$. For a multi-index
$\alpha=(\alpha_1,\dots,\alpha_d)\in {\bf N}^d$ we write
$\partial^{\alpha}=\partial^{\alpha_1}\dots\partial^{\alpha_d}$ and
$|\alpha|=|\alpha_1|+\dots |\alpha_d|$. We will often use Stirling's formula:
$$
N!=N^N e^{-N}\sqrt{2\pi N}e^{\theta_N \over 12N},
$$
for some $0<\theta_N<1$, $N\in\Z.$ By $C^{\infty}(K)$
we denote the
set of smooth functions on a compact set $K\subset\subset U$ with smooth boundary, where $U
\subseteq \Rd$ is an open set, $C_K^{\infty}$ are smooth functions supported by $K$.
The closure of $ U\subset \Rd$ is denoted by $ \overline{U}$.
A conic neighborhood of $\xi_0 \in \Rd  \setminus 0$  is an open cone
$\Gamma \subset \Rd $ such that  $ \xi_0 \in \Gamma$.
The Fourier transform of a  locally integrable function $f$ is defined as
$ \widehat{f}(\xi)=\int_{{\Rd}}f(x)e^{-2\pi i x\xi}dx$, $\xi \in \Rd$,
and the definition extends to distributions by duality.
Open ball of radius $r$, centered at $x_0\in \Rd$ is denoted by $B_r(x_0)$.

\par

For locally convex topological spaces   $X$ and $Y$,
$X\hookrightarrow Y$ means that $X$ is dense in $Y$ and that the identity mapping
from $X$ to $Y$ is continuous, and we use
$ \varprojlim $  and $\varinjlim $ to denote the projective and inductive limit topologies
respectively.
By $X'$ we denote the strong dual of $X$ and by
$\langle \cdot, \cdot \rangle_{X}$ the dual pairing between $X$ and
$X'$.
As usual,  $\D'(U)$ stands for Schwartz distributions, and $\E'(U)$ for compactly supported distributions.

\subsection{Ultradifferentiable functions and wave-front sets} \label{definicije}

For the sake of the clarity of our exposition, in this subsection we recall
Komatsu's approach to the theory of ultradifferentiable functions, see \cite{Komatsuultra1},
and the notion of wave-front set in the context of the Gevrey regularity.

By $M_p = (M_p)_{p\in \N}$ we denote a sequence of positive numbers
such that the following conditions hold:
\begin{eqnarray*}
(M.0) & M_0=1; & \\
(M.1) & M_p ^2 \leq M_{p-1}M_{p+1}, & p\in \Z; \\
(M.2) & (\exists C> 0) \;\; M_{p+q}\leq C^{p+1} M_p M_q, & p,q\in \N; \\
(M.3)' & \suml_{p=1}^{\infty}\frac{M_{p-1}}{M_p}<\infty.
\end{eqnarray*}

Then $M_p$ also satisfies weaker conditions: $(M.1)'$ $M_p M_q\leq
M_{p+q}$ and $(M.2)' \; $  $(\exists C>0 )\;\;$ $M_{p+q}\leq C_q^{p+1} M_p$,
$p,q\in \N$.

\par

Let the sequence $M_p $  satisfy the conditions  $(M.0)-(M.3)'$ and
let $ U \subseteq \Rd$ be an open set. A function $\phi \in
C^{\infty}(U)$ is an {\em ultradifferentiable function of class} $(M_p)$
(resp. {\em of class} $ \{ M_p \} $) if for each
compact subset $ K\subset\subset U$ and each $h>0$, there exists
$C > 0$
(resp. for each
compact subset $ K\subset\subset U$ there exists
$h>0$ and
$C > 0$)
such that

\be \label{estimate-ultradiff}
\displaystyle \sup_{x\in K}
|\partial^{\alpha} \phi (x)| \leq C h^{|\alpha|}M_{|\alpha|}, \;\;\;
\alpha \in \N^d.
\ee

For a fixed compact set  $K \subset \Rd $ and  $h>0$,
$\phi \in \E^{\{M_p\},h}(K)$ if $\phi \in C^{\infty}(K)$  and if
\eqref{estimate-ultradiff} holds for some $C> 0$. If $\phi \in
C^{\infty}(\Rd)$ and $ \supp \phi \subset K, $ then $\phi \in
D_K^{\{M_p\},h} $. These spaces are Banach spaces under the norm
$$
\| \phi \|_{\E^{\{M_p\},h}(K)} =
\sup_{\alpha \in \N^d, x\in K} \frac{|\partial^{\alpha} \phi (x)|}{h^{|\alpha|}M_{|\alpha|}}.
$$

The spaces of ultradifferentiable functions of class  $ \{M_p\} $ and of class  $ (M_p) $
are respectively given by
$$
\E^{\{M_p\}}(U)=\varprojlim_{K\subset\subset U}\varinjlim_{h\to \infty}\E^{\{M_p\},h}(K)
= \bigcap_{K\subset\subset U} \bigcup_{h\to \infty}\E^{\{M_p\},h}(K),
$$
$$
\E^{(M_p)}(U)=\varprojlim_{K\subset\subset U}\varprojlim_{h\to 0}\E^{\{M_p\},h}(K)
= \bigcap_{K\subset\subset U} \bigcap_{h\to 0}\E^{\{M_p\},h}(K),
$$
and their strong duals are respectively called the space of ultradistributions of Roumieu type of class $M_p$
and the space of ultradistributions of Beurling type of class $M_p$.

The space of ultradifferentiable functions of class  $ \{M_p\} $ (resp. of class  $ (M_p) $)
with support in $K$ is given by
$$
\quad \D^{\{M_p\}} (U) =\varinjlim_{K\subset\subset U}\varinjlim_{h\to \infty}\D_K^{\{M_p\},h}
= \bigcup_{K\subset\subset U}\bigcup_{h\to \infty}\D_K^{\{M_p\},h}
$$
$$
\text{ (resp.}
\quad \D^{(M_p)} (U) =\varinjlim_{K\subset\subset U}\varprojlim_{h\to 0}\D_K^{\{M_p\},h}
= \bigcup_{K\subset\subset U}\bigcap_{h\to 0}\D_K^{\{M_p\},h} \text{)}
$$
and its strong dual is the space of compactly supported ultradistributions  of Roumieu type of class $M_p$
(resp. of Beurling type of class $M_p$).

In what follows, $ \E^{*}(U) $ and $ \D^{*} (U)$ stand for
$ \E^{\{M_p\}}(U)$ or $ \E^{(M_p)}(U)$, and for  $\D^{\{M_p\}} (U)$ or
$ \D^{(M_p)} (U)$, respectively.

In particular, if $M_p$ is the {\em Gevrey sequence}, $M_p=p!^t$, $t>1$, then $ \E^{\{p!^t\}}(U)$ and
$ \E^{(p!^t)}(U)$  are the {\em Gevrey classes} of ultradifferentiable functions commonly denoted by
${\E}_{t}(U)$. Note that $ p!^t $, $t>1$,
satisfies $(M.0)-(M.3)'$.
We refer to \cite{Komatsuultra1} for a detailed study of different classes
of ultradifferentiable functions and their duals.

\par

Next we recall the notion of wave-front set in the context of the Gevrey regularity.

Let there be given $t\geq 1$ and $(x_0,\xi_0)\in U\times \Rd\backslash\{0\}$.
Then the {\em Gevrey wave front set} $WF_t(u)$ of $u\in \D'(U)$
can be defined as follows: $(x_0,\xi_0)\not\in WF_t (u)$ if and only if there exists an open neighborhood
$\Omega$ of $x_0$, a conic neighborhood $\Gamma$ of $\xi_0$ and a bounded sequence
$u_N\in \E'(U)$, such that $u_N=u$ on $\Omega$ and
$$
|\widehat u_N(\xi)|\leq A\, \frac{h^{N}N!^t }{|\xi|^{ N }}, \quad N\in { \Z},\,\xi\in\Gamma,
$$
for some $A,h>0$. The wave-front set of $u\in \D'(U)$ can be defined in an analogous way.
If $t=1$, then the Gevrey wave front set is sometimes called the {\em analytic wave front set}
and denoted by  $WF_A(u)$, $u\in \D'(U)$. The classical $C^\infty $ wave-front set of
$u\in \D'(U)$ can be also defined through its complement:
$(x_0,\xi_0)\not\in WF (u)$ if and only if there exists an open neighborhood
$\Omega$ of $x_0$, a conic neighborhood $\Gamma$ of $\xi_0$ and a smooth compactly supported function
$\phi$, equal to $1$ on $\Omega$ and
$$
|\widehat \phi u (\xi)|\leq  \frac{C_N}{(1+|\xi|)^{ N }}, \quad N\in { \Z},\,\xi\in\Gamma, C_N>0.
$$
We refer to \cite{HermanderKnjiga, Rodino, Foland} for details.

\section{$M_p^{\t,\s}$ sequences and the corresponding regularity classes} \label{sec-1}

In this section we observe the sequence
$M_p^{\t,\s}=p^{\t p^{\s}}$, $p\in \N$, where $\t>0$, $\s>1$ and
study its basic properties. Although $M_p^{\t,\s}$ fails to satisfy the condition (M.2),
the flexibility obtained by introducing the two-parameter dependence
enables us to introduce smooth  functions which are less regular than the Gevrey functions.
In a separate subsection we define ultradifferentiable functions of class $M_p^{\t,\s}$ and study their main
properties.

\subsection{The defining sequence $M_p^{\t,\s}$}\label{SekcijaOsobineNiza}

The following lemma captures the basic properties of the sequence
$M_p^{\t,\s}=p^{\t p^{\s}}$, $p\in \N$, $\t>0$, $\s>1$, $M_0 ^{\t,\s}= 1$.
We refer to \cite{PTT-01} for the proof.

\begin{lema} \label{osobineM_p_s}
Let $\tau>0$, $\s>1$ and $M_p^{\tau,\s}=p^{\tau p^{\s}}$, $p\in \Z$, $M_0^{\tau,\s}=1$.
Then, apart from  $(M.1)$ and $(M.3)'$ the sequence
 $M_p^{\tau,\s}$ satisfies the following properties.

$\widetilde{(M.2)'}$
$M_{p+q}^{\tau,\s}\leq C_q^{p^{\s}}M_p^{\tau,\s}$, for some  sequence $C_q\geq 1$, $p,q\in \N$,  \medskip

$\widetilde{(M.2)}$ $M_{p+q}^{\tau,\s}\leq
C^{p^{\s} + q^{\s}}M_p^{\tau 2^{\s-1},\s}M_q^{\tau 2^{\s-1},\s}$,
$p,q\in \N$, for some constant $ C>1$.

Furthermore, there exist $A,B,C>0$ such that
$$
M_p^{\tau,\s}\leq A C^{p^{\s}}{\lf p^{\s}
\rf}!^{\tau/\s}\quad  and \quad {\lf p^{\s}  \rf}!^{\tau/\s}\leq B
M_p^{\tau,\s}.
$$

\end{lema}

The property $(M.3)'$ implies that the corresponding spaces of ultradifferentiable functions (see Subsection \ref{subsec-classes})
are non-quasianalytic, i.e. they contain nontrivial compactly supported smooth functions. Let us now prove $(M.3)'$
by modifying the proof given in \cite{PTT-01}. Since the first summand in the series given in
$(M.3)'$  is equal to $1$, it is enough to observe the summation for $p \geq 2$.
Since $p^{\s}\geq (p-1)^{\s-1}p=(p-1)^{\s}+(p-1)^{\s-1}$, $p\in \Z$, we have
$$
\suml_{p=2}^{\infty}\frac{(p-1)^{\tau (p-1)^{\s}}}{p^{\tau p^{\s}}}
\leq
\suml_{p=2}^{\infty}
\frac{(p-1)^{\tau (p-1)^{\s}}}{ p^{\tau ( (p-1)^{\s} + (p-1)^{{\s}-1}) }}
 = \suml_{p=2}^{\infty}
\frac{ (1 - \frac{1}{p})^{\tau (p-1)^{\s}}}{p^{\tau (p-1)^{{\s}-1} }}.
$$
Since
$ \displaystyle
\tau p^{{\s}} \ln \Big(1+\frac{1}{p}\Big) = {\tau}\, p^{{\s}-1} \ln \Big(1+\frac{1}{p}\Big)^p \geq {\tau}\, p^{{\s}-1} \ln 2,$
$ p \geq 2, $ we obtain
$$
2^{\tau p^{{\s}-1}}\leq \Big(1+\frac{1}{p}\Big)^{\tau p^{{\s}}},\;\;\; p \geq 2,
$$
which gives
$$
\suml_{p=2}^{\infty}
\frac{ (1 - \frac{1}{p})^{\tau (p-1)^{\s}}}{p^{\tau (p-1)^{{\s}-1} }}
\leq
\suml_{p=2}^{\infty} \frac{1}{{(2p)^{\tau {(p-1)^{{\s}-1}}}}}   <\infty.
$$

\subsection{Regularity classes $\E_{\t,\s}$} \label{subsec-classes}

Let $\tau>0$, $\s>1$, $h>0$, and
$K\subset\subset U$, where $U$ is an open set in $\Rd$. A function $\phi \in C^\infty (U)$
is ultradifferentiable function of class $M_p^{\t,\s}$ if there exists $A>0$
such that
$$ \displaystyle |\partial^{\alpha}\phi(x)|\leq A
h^{|\alpha|^{\s}}|\alpha|^{\t|{\alpha}|^{\s}},\quad \alpha \in \N^d, x\in K.
$$
The space of  ultradifferentiable functions of class $M_p^{\t,\s}$ denoted  by
${\E}_{\t, {\s},h}(K)$ is a Banach space with the norm given by
\begin{equation} \label{Norma}
\| \phi \|_{{\E}_{\t, {\s},h}(K)}=\sup_{\alpha \in \N^d}\sup_{x\in K}
\frac{|\partial^{\alpha} \phi (x)|}{h^{|\alpha|^{\s}}|\alpha|^{\t|\alpha|^{\s}}}\,,
\end{equation}
and $ \displaystyle
{\E}_{\t_1, {\s_1},h_1}(K)\hookrightarrow {\E}_{\t_2,
{\s_2},h_2}(K)$, $0<h_1\leq h_2,$ $0<\t_1\leq\t_2$, $1<\s_1\leq \s_2.$
By \eqref{niz1}, the norm in \eqref{Norma} is equivalent to
$$
\|\phi\|^{\sim}_{{\E}_{\tau, {\s},h}(K)}=\sup_{\alpha \in
\N^d}\sup_{x\in K}\frac{|\partial^{\alpha} \phi
(x)|}{h^{|\alpha|^{\s}}\lf|\alpha|^{\s}\rf!^{\tau/\s}}<\infty,\quad
h>0.
$$

Let ${\D}^K_{\t, \s,h}$ be the set of functions in ${\E}_{\t,
\s,h}(K)$ with support contained in $K$. Then, in the topological sense,
we set
$$
{\E}_{\{\t,
\s\}}(U)=\varprojlim_{K\subset\subset U}\varinjlim_{h\to
\infty}{\E}_{\t, {\s},h}(K),
$$
$$
{\E}_{(\t,
\s)}(U)=\varprojlim_{K\subset\subset U}\varprojlim_{h\to 0}{\E}_{\t,
{\s},h}(K),
$$
$$
{\D}_{\{\t,
\s\}}(U)=\varinjlim_{K\subset\subset U} {\D}^K_{\{\t, \s\}}
=\varinjlim_{K\subset\subset U} \varinjlim_{h\to\infty}{\D}^K_{\t,
\s,h}\,,
$$
$$
{\D}_{(\t,
\s)}(U)=\varinjlim_{K\subset\subset U} {\D}^K_{(\t, \s)}
=\varinjlim_{K\subset\subset U} \varprojlim_{h\to 0}{\D}^K_{\t,
\s,h}.
$$
We will use abbreviated notation $ \t,\s $ for
$\{\t,\s\}$ or $(\t,\s)$ .

\begin{rem} 
If $ \t > 1 $ and $\s = 1$,
then $ {\E}_{\t, 1}(U)={\E}_{\t}(U)$  are the Gevrey classes
and $\D_{\t,1}(U)=\D_{\t}(U)$ are the corresponding subspaces of compactly supported functions in
$\E_{\t}(U)$. When $0<\t\leq 1$ and $ \s = 1$ such spaces are contained in the corresponding spaces of
quasianalytic functions. In particular, $\dss
\D_{\t}(U)=\{0\}$ when $0<\t\leq 1$.
\end{rem}

\par

Clearly, compactly supported Gevrey functions belong to ${\D}_{\{\t, \s\}}(U)$.
However, one can find a compactly supported
function in ${\D}_{\{\t, \s\}}(U)$  which is not in
${\D}_\t(U)$, for any $\t>1$.
We refer to \cite{PTT-01} for the proofs.

\par

It is known that the spaces $\E^{\{M_p\}}(U)$ are nuclear if the defining sequence  $M_p$ satisfies
$(M.2)'$, cf. \cite[ Theorem 2.6 ]{Komatsuultra1}. Although $M_p^{\t,\s}=p^{\t p^{\s}}$, $\t>0$, $\s>1$,
does not satisfy $(M.2)'$, it can be proved that
the spaces ${\E}_{\t, \s}(U)$, ${\D}^K_{\t, \s}$ and ${\D}_{\t, \s}(U)$
are nuclear as well. Again, we refer to  \cite{PTT-01} for the proof.

\par

The basic embeddings between the introduced spaces with respect to $\s$ and $\t$ are given in the following proposition.

\begin{prop}
\label{detectposition} Let $\s_1\geq 1$. Then for every $\s_2>\s_1$
and $\t>0$
\begin{equation}
\label{Theta_S_embedd} \varinjlim_{\t\to \infty}{\E}_{\t,
{\s_1}}(U)\hookrightarrow \varprojlim_{\t\to 0^+} {\E}_{\t,
{\s_2}}(U).
\end{equation} Moreover, if $0<\t_1<\t_2$, then for every $\s\geq 1$ it holds
\be \label{RoumieuBeurling} \E_{\{\t_1,\s\}}(U)\hookrightarrow
\E_{(\t_2,\s)}(U)\hookrightarrow \E_{\{\t_2,\s\}}(U), \ee
and
$$
\varinjlim_{\t\to \infty}{\E}_{\{\t, {\s}\}}(U)= \varinjlim_{\t\to \infty} {\E}_{(\t, {\s})}(U),
\;\;\;
\varprojlim_{\t\to 0^+}{\E}_{\{\t, {\s}\}}(U)= \varprojlim_{\t\to 0^+} {\E}_{(\t, {\s})}(U).
$$
\end{prop}

For the proof of  \eqref{Theta_S_embedd} we refer to \cite{PTT-01} and complete proof of Proposition \ref{detectposition} can be found in \cite{PTT-02}.

We are also interested in projective
(when $\t \rightarrow 0^+ $ or when $\s \rightarrow 1^+$) and inductive
(when $\t \rightarrow \infty $  or when  $\s \rightarrow \infty$)
limits which we  denote as follows:
$$
\E_{0,\s}(U):=\varprojlim_{\t\to 0^+}\E_{\t,{\s}}(U), \;\;\;
\E_{\infty,\s}(U):=\varinjlim_{\t\to \infty}\E_{\t,{\s}}(U),
$$
$$
\E_{\t,1}(U):=\varprojlim_{\s\to 1^+}\E_{\t,{\s}}(U), \;\;\;
\E_{\t,\infty}(U):=\varinjlim_{\s\to \infty}\E_{\t,{\s}}(U),
$$
\be
\label{BorderLinePresek}
\E_{0,1}(U):= \varprojlim_{\s\to 1^+}\E_{0,\s}(U), \;\;\;
\E_{0,\infty}(U):=\varinjlim_{\s\to \infty}\E_{0,\s}(U),
\ee
\be
\label{BorderLineUnija}
\E_{\infty,1}(U):=\varprojlim_{\s\to 1^+}\E_{\infty,\s}(U), \;\;\;
\E_{\infty,\infty}(U):=\varinjlim_{\s\to \infty}\E_{\infty,\s}(U),
\ee
The proof of the following corollary can be found in \cite{PTT-02}.
\begin{cor}
With the notation from \eqref{BorderLinePresek} and \eqref{BorderLineUnija} the following strict embeddings
hold true:
$$
\varinjlim_{t\to\infty}\E_t(U)\hookrightarrow
{\E}_{0, 1}(U)\hookrightarrow {\E}_{\infty, 1}(U)
\hookrightarrow \E_{0,\infty}(U) \hookrightarrow
\E_{\infty,\infty}(U).
$$
\end{cor}

Recall that the Komatsu's condition  $(M.2)'$, also known as "stability under differential operators",
is sufficient to ensure that the spaces ${\E}^*(U)$ are closed under the differentiation, cf. \cite[Theorem 2.10]{Komatsuultra1}.
In the next Proposition we show that  ${\E}_{\t, \s}(U)$ is closed
under the finite order differentiation, although the condition  $(M.2)'$ is violated.

\par

\begin{te} \label{tealgebra}
Let $U$ be open in $\Rd$, and let  $\t>0$ and $\s>1$.
Then the space ${\E}_{\t, \s}(U)$ is closed under pointwise
multiplications and finite order differentiation.
\end{te}

\begin{proof} We leave to the reader to prove that the spaces are closed under
translations and dilations and show the algebra property first.

Let $K\subset\subset\Rd$ and for $h>0$ set
$c_h=\min\{h,h^{2^{\s-1}}\}$. Then for $\phi,\psi\in {\E}_{\t,
\s,c_h}(K)$, by the Leibnitz formula we obtain
\begin{eqnarray*}
||\phi\psi||_{{\E}_{\t, \s,2h}(K)}&\leq&
\sup_{\alpha\in \N^d}\sum_{\beta \leq \alpha} {\alpha \choose \beta}
\frac{c_h^{{|\alpha-\beta|}^{\s}}c_h^{|\beta|^{\s}}|\alpha-\beta|^{\t|\alpha-\beta|^{\s}}|\beta|^{\t |\beta|^{\s}}}{(2h)^{|\alpha|^{\s}}|\alpha|^{\t|\alpha|^{\s}}}
\nonumber\\
&\cdot&||\phi||_{{\E}_{\t, \s,c_h}(K)}||\psi||_{{\E}_{\t,
\s,c_h}(K)}. \label{proizvod-estimate}
\end{eqnarray*}

If $h\geq 1$, then we put $c_h=h$. Note that
$|\alpha-\beta|^{\s}+|\beta|^{\s}\leq |\alpha|^{\s}$ when $\beta\leq \alpha$.
By $(M.1)$ property of $M_p^{\t,\s}$ we then have
$$
\sum_{\beta \leq \alpha} {\alpha \choose \beta}\frac{c_h^{{|\alpha-\beta|}^{\s}}
c_h^{|\beta|^{\s}}|\alpha-\beta|^{\t|\alpha-\beta|^{\s}}|\beta|^{\t |\beta|^{\s}}}{(2h)^{|\alpha|^{\s}}|\alpha|^{\t|\alpha|^{\s}}}
\leq \frac{2^{|\alpha|} h^{|\alpha|^{\s}}}{(2h)^{|\alpha|^{\s}}}\leq 1,\quad \alpha\in \N^d.
$$

If $0<h<1$, then $c_h=h^{2^{\s-1}}$, and
$$
(1/h)^{|\alpha|^{\s}}\leq
(1/h)^{2^{\s-1}|\alpha-\beta|^{\s}}(1/h)^{2^{\s-1}|\beta|^{\s}},
\;\;\;
\beta\leq \alpha.
$$
which gives
$$\sum_{\beta \leq \alpha} {\alpha \choose \beta}
\frac{c_h^{{|\alpha-\beta|}^{\s}}c_h^{|\beta|^{\s}}|\alpha-\beta|^{\t|\alpha-\beta|^{\s}}|\beta|^{\t |\beta|^{\s}}}{(2h)^{|\alpha|^{\s}}|\alpha|^{\t|\alpha|^{\s}}}
\leq \frac{2^{|\alpha|}}{2^{|\alpha|^{\s}}} \leq 1,\quad \alpha\in \N^d,
$$
that is
$$
\|\phi\psi \|_{{\E}_{\t, \s,2h}(K)}
\leq \|\phi \|_{{\E}_{\t, \s,c_h}(K)} \|\psi \|_{{\E}_{\t,
\s,c_h}(K)},
$$
wherefrom the algebra property holds.

\par

To  prove that  ${\E}_{\t, \s}(U)$ is closed under differentiation we fix $\beta\in
\N^d$, and set $c_h'=\max\{h,h^{2^{\s-1}}\}$, $h>0$. Then, for
every $x\in K$, from $\widetilde{(M.2)'}$ it follows that
\begin{eqnarray}
|(\partial^{\alpha+\beta}\phi(x))|&\leq & ||\phi||_{{\E}_{\t, \s,h}(K)}
h^{|\alpha+\beta|^{\s}}|\alpha+\beta|^{\t|\alpha+\beta|^{\s}} \nonumber \\
&\leq&
||\phi||_{{\E}_{\t, \s,h}(K)} C_h'^{|\beta|^{\s}}
(C_{|\beta|}c_h')^{|\alpha|^{\s}}|\alpha|^{\t|\alpha|^{\s}}\,,  \nonumber
\end{eqnarray}
where $C_h'=\max\{1,h^{2^{\s-1}}\}$ and $C_{|\beta|}$
is the constant from  $\widetilde{(M.2)'}$
(see Lemma \ref{osobineM_p_s} for $q=|\beta|$).
This implies that for every $h>0$ there exists $C_{h,\beta}>0$ such that
$$\dss||\partial^{\beta}\phi||_{{\E}_{\t, \s,C_{|\beta|}c_h'}(K)}\leq C_{h,\beta}||\phi||_{{\E}_{\t, \s,h}(K)},$$
and the statement follows.
\end{proof}

We conclude this section with the continuity properties of certain
ultradifferentiable operators $P(x,\partial)$ acting on
${\E}_{\tau, \s}(U)$. Note that Komatsu's condition $(M.2)$
provides the stability of  ${\E}^*(U)$ under the action of ultradifferentiable operators,
cf. \cite[Theorem 2.12]{Komatsuultra1}.
The following theorem shows that the condition $\widetilde{(M.2)}$ provides instead only
the continuity of certain ultradifferentiable operators from
${\E}_{\t, \s}(U) $ into $ {\E}_{\t2^{\s-1}, \s}(U)$.
We refer to \cite[Theorem 2.1]{PTT-02} for a more general result which involves
ultradifferentiable operators with non constant coefficients.

\begin{te}
\label{TeoremKonstant} Let $U$ be open in $\Rd$, $\t>0$ and $\s>1$.
If
$P(\partial)=\suml_{|\alpha|=0}^{\infty}a_{\alpha}{\partial}^{\alpha}$
is a constant coefficient differential operator of infinite order
such that for some $L>0$ and $A>0$  (resp. every $L>0$ there exists
$A>0$) such that \be \label{constantKoef} |a_{\alpha}|\leq A
\frac{L^{|\alpha|^{\s}}}{|\alpha|^{\t 2^{\s-1}{|\alpha|}^{\s}}}, \ee
then ${\E_{\infty,\s}(U)}$  is
closed under action of $P(\partial)$. In particular,

$$P(\partial):\quad {\E}_{\t, \s}(U) \longrightarrow {\E}_{\t
2^{\s-1}, \s}(U)\,, $$ is continuous linear mapping, where
${\E}_{\t, \s}(U)$ denotes ${\E}_{(\t, \s)}(U)$ or ${\E}_{\{\t,
\s\}}(U)$.
\end{te}

\begin{proof}
Let $\phi\in {\E}_{\t, \s,h}(K)$, for some $h>0$.
Then, for $x\in K$, using \eqref{constantKoef} and $\widetilde{(M.2)}$ property of $M_p^{\t,\s}$ we obtain
\begin{eqnarray}
\label{constultradif}
|\partial^{\beta}(a_{\alpha}\partial^{\alpha}\phi(x))|&\leq&
 A ||\phi||_{{\E}_{\t, \s,h}(K)} \frac{L^{|\alpha|^{\s}}}{|\alpha|^{\t 2^{\s-1}{|\alpha|}^{\s}}}
 h^{{|\alpha+\beta|}^{\s}}(|\alpha+\beta|)^{{\t |\alpha+\beta|}^{\s}}\nonumber\\
&\leq & A ||\phi||_{{\E}_{\t, \s,h}(K)} \frac{L^{|\alpha|^{\s}}}{|\alpha|^{\t 2^{\s-1}{|\alpha|}^{\s}}}h^{|\alpha+\beta|^{\s}}C^{|\alpha|^{\s}}C^{|\beta|^{\s}}{|\alpha|^{\t 2^{\s-1}{|\alpha|}^{\s}}}|\beta|^{\t 2^{\s-1}|\beta|^{\s}}\nonumber\\
& \leq & A ||\phi||_{{\E}_{\t, \s,h}(K)} (L C c_h)^{|\alpha|^{\s}}(C
c_h)^{|\beta|^{\s}}|\beta|^{\t 2^{\s-1}|\beta|^{\s}} ,
\end{eqnarray} where for the last inequality we have used that for $\s>1$,
$$|\alpha|^{\s}+|\beta|^{\s}\leq|\alpha+\beta|^{\s}\leq 2^{\s-1}(|\alpha|^{\s}+|\beta|^{\s}),$$
$c_h=\max\{h,h^{2^{\s-1}}\}$ and $C>1$ is the constant from $\widetilde{(M.2)}$.
Note that $c_h=h$ when $0<h\leq 1$ and $c_h=h^{2^{\s-1}}$ when $h>1$.

Now, we may choose either $h>0$ or   $L>0$ such that $ L C c_h< 1/2 $
holds true.
Since $\dss
\sum_{|\alpha|=0}^{\infty}(1/2)^{|\alpha|^{\s}} < \infty $,
by taking the sum with respect to $\alpha$ and the supremum with
respect to $\beta$ and $x\in K$, from (\ref{constultradif}) it
follows that $$
||P(\partial)\phi||_{{\E}_{\t2^{\s-1}, \s,C c_h}(K)}\leq C'
||\phi||_{{\E}_{\t, \s,h}(K)}, $$ for some $C'>0$ and the theorem
is proved, since the result for ${\E_{\infty,\s}(U)}$  follows immediatelly.
\end{proof}

\section{Local analysis of distributions with respect to $\E_{\t,\s}$} \label{sec-2}

In this section we study local behavior of distributions
by the means of appropriate localized versions of their Fourier transforms.
The localization is defined by the means of  $\t,\s$-admissible sequences of smooth functions,
see Definition \ref{definicijaNiza}. We first describe the process of enumeration, which is one of the main tolls in our analysis.

\par

Let $\t>0$, $\s>1$, $\Omega\subseteq K\subset \subset U\subseteq\Rd$, where $\Omega$ and $U$ are open in $\Rd$, and the closure of
$\Omega$ is contained in $K$.
Let $u\in \D'(U)$. We observe the nature of its regularity with respect to the condition
 \begin{equation}
\label{uslov3'}
|\widehat u_N(\xi)|\leq A\, \frac{h^{N} N!^{\t/\s}}{|\xi|^{\lfloor N^{1/\s} \rfloor}}, \quad N\in { \N},\,\xi\in \Rd\backslash\{0\},
\end{equation}
where $\{u_N\}_{N\in \N}$ is bounded sequence in $\E'(U)$ such that $u_N=u$ in $\Omega$ and $A,h$ are some positive constants.

\par

One of the main ingredients of the following proofs is
the procedure which we call {\em enumeration}
and which consists of  a change of variables in indices which "speeds up" or "slows down"
the decay estimates of single members of the corresponding sequences,
while preserving their asymptotic behavior when $N \rightarrow \infty$. In other words,
although estimates for terms of a sequence before and after
enumeration are different, the  asymptotic behavior of the whole sequence remains unchanged.
Therefore, the condition (\ref{uslov3'})
is equivalent to another condition obtained after replacing $N$ with positive, increasing sequence $a_N$
such that $a_N\to \infty$, $N\to \infty$. We then  write $N\to a_N$ and $u_N$ instead of $u_{ a_N }$.

For example, applying Stirling's formula to (\ref{uslov3'}) we obtain
\begin{equation}
\label{uslov3''}
|\widehat u_N(\xi)|\leq A_1\, \frac{h_1^{N} N^{\frac{\t}{\s} N} }{|\xi|^{\lfloor N^{1/\s} \rfloor}}, \quad N\in { \N},\,\xi\in \Rd\backslash\{0\},
\end{equation}
for some positive constants $A_1,h_1$. After enumeration $N\to N/\t$ and writing $u_N$
instead of the $u_{N/\t}$, (\ref{uslov3''}) becomes
$$
|\widehat u_N(\xi)|\leq A_1\, \frac{h_1^{N/\t} (N/\t)^{\frac{\t}{\s} (N/\t)} }{|\xi|^{\lfloor (N/\t)^{1/\s} \rfloor}}
\leq A_2\, \frac{h_2^{N} N!^{1/\s}}{|\xi|^{\lfloor (N/\t)^{1/\s} \rfloor}}, \quad N\in { \N},\,\xi\in \Rd\backslash\{0\}
$$
for some $A_2,h_2>0$. Moreover, if $\{u_N\}_{N\in \N}$ is  bounded sequence in $\E'(U)$, then $\{u_{N/\t}\}_{N\in \N}$
is also bounded in $\E'(U)$ (with respect to the strong topology).

In  \cite[Proposition 8.4.2]{HermanderKnjiga} H\"ormander used a sequence of carefully chosen
cutoff functions $\{\chi_N\}_{N\in \N}$ to define the analytic wavefront set $WF_A$.
We modify that to define and analyze a new type of wavefront sets in $\D'(U)$ related to \eqref{uslov3'} or
\eqref{uslov3''}.

\begin{de}
\label{definicijaNiza}
Let $\t>0$, $\s>1$, and $\Omega\subseteq K\subset\subset U$,
such that $\overline{\Omega} $ is strictly contained in $ K$.
A sequence $\{\chi_N\}_{N\in \N}$ of functions in $C^{\infty}_K$
is said to be {\em  $\t,\s$-admissible  with respect to $K$} if
\begin{itemize}
\item[a)] $\chi_N=1$ in a neighborhood of $\Omega$, for every $N\in \N$,
\item[b)] there exists a positive sequence $C_{\beta}$ such that
\begin{equation}
\label{ocenaNiz}
\sup_{x\in K}
|D^{\alpha+\beta}\chi_{N}(x)|\leq C_{\beta}^{|\alpha|+1}
\lfloor N^{1/\s}\rfloor ^{|\alpha|},
\quad |\alpha|\leq
\lfloor (N/\t)^{1/\s}  \rfloor,
\end{equation} for every $N \in \N$ and  $\beta\in {\N}^d$.
\end{itemize}
\end{de}

\begin{rem}
When  $\t=\s=1$ we recover the sequence $\{\chi_N\}_{N\in \N}$
used by H\"ormander to analyze the analytic behavior of distributions.
Moreover, note that for $\s>1$ and $0<\t\leq 1$, $\{\chi_{\t N^{\s}}\}_{N\in \N}$
gives another sequence with the same asymptotic properties as  $\{\chi_N\}_{N\in \N}$.
This implies that, for $\s>1$ and $0<\t\leq 1$, the enumeration $N\to \t N^{\s}$
in (\ref{ocenaNiz}) may be used to define the analytic wave-front sets.
\end{rem}

\begin{rem}
\label{PaleyRemark}
From  \eqref{ocenaNiz} it follows that
\be
\label{chiNPaley}
|\widehat \chi_N (\xi)|\leq A_{\beta}^{|\alpha|+1}\lfloor N^{1/\s}\rfloor ^{|\alpha|}\langle\xi\rangle^{-|\alpha|-|\beta|},\quad
|\alpha|\leq
\lfloor (N/\t)^{1/\s}  \rfloor,
\ee
for every $N \in \N, \xi\in \Rd$, where $\langle\xi\rangle=(1+|\xi|^2)^{1/2}$.
If $\alpha=0$ in \eqref{ocenaNiz}, then $\{\chi_N\}_{N\in \N}$ is bounded sequence in $C^{\infty}(U)$
and from \eqref{chiNPaley}, it follows that $\{\widehat\chi_N\}_{N\in \N}$ is bounded in the
Schwartz space ${\mathcal S}(\Rd)$.
From the boundedness of $\{\chi_N\}_{N\in \N}$ in $C^{\infty}(U)$, it follows that
$\{\chi_N u\}_{N\in \N}$ is bounded in $\E'(U)$ for every  $u\in \D'(U)$ .

Recall that if $\{u_N\}_{N\in \N}$ is a bounded sequence in $\E'(U)$ then
Paley-Wiener type theorems and $e^{-ix\cdot\xi}\in C^{\infty}(\Rd_x)$, for every $\xi\in \Rd$, imply
\begin{equation}
\label{ogranicenostUN}
|\widehat u_N(\xi)|=|\langle u_N,e^{-i\cdot\xi}\rangle|\leq C\langle\xi\rangle^M,
\end{equation} for some $C,M>0$ independent of $N$.
\end{rem}

The existence of the $\t,\s$-feasible sequences is
shown in the following Lemma.
We refer to \cite{PTT-02} for the proof, see also \cite[Theorems 1.3.5 and 1.4.2]{HermanderKnjiga}.

\begin{lema}
\label{ogranicenostNiza}
Let there be given $r>0$, $\t>0$, $\s>1$ and $x_0\in \Rd$.
There exists $\t,\s$-admissible sequence $\{\chi_N\}_{N\in \N}$ with respect to
$\overline{B_{2r}(x_0)}$ such that $\chi_N=1$ on $B_{r}(x_0)$, for every $N\in \N$.
\end{lema}

Next we show that (\ref{uslov3'})
implies local regularity related to  $\E_{\{\t,\s\}}(U)$.
For the opposite direction, if $u\in \E_{\{\t,\s\}}(\Omega)$
we need to observe $ \tilde \t, \s $-admissible sequences, where $ \tilde \t = \t^{\s/(\s -1)}.$
The precise statements are the following.

\begin{prop}
\label{dovoljanUslov}
Let $u \in \D'(U)$, $\t>0$, $\s>1$, $\Omega\subseteq U$ with the closure contained in $U$ and let $\{u_N\}_{N\in \N}$ be a bounded sequence in $\E'(U)$,
$u_N=u$ on $\Omega$ and  such that (\ref{uslov3'}) holds. Then $u\in \E_{\{\t,\s\}}(\Omega)$.
\end{prop}
\begin{proof}
After the enumeration $N\to N^{\s}$ and by Lemma \ref{osobineM_p_s},
condition (\ref{uslov3'}) is equivalent to

\begin{equation}
\label{uslov4}
|\widehat u_{N}(\xi)|\leq A\, \frac{k^{N^{\s}} N^{\t N^{\s}} }{|\xi|^{N}}, \quad N\in { \N},\,\xi\in \Rd\backslash\{0\}.
\end{equation} for some $A,k>0$.

By the Fourier inversion formula and the fact that $u_N=u$ in $\Omega$ we obtain
\begin{multline}
(h^{|\alpha|^{\s}}|\alpha|^{\t |\alpha|^{\s}})^{-1}|D^{\alpha}u(x)| \\[1ex]
= (h^{|\alpha|^{\s}}|\alpha|^{\t |\alpha|^{\s}})^{-1}\Big|\Big(
\int_{|\xi|\leq 1}+\int_{|\xi|> 1}\Big)\xi ^{\alpha}\widehat u_N(\xi)e^{2\pi i x\xi}d\xi\Big|\\[1ex]
\leq I_1+I_2,\quad N\in \N, \alpha \in \N^d, x\in \Omega,\nonumber
\end{multline}
where $h>0$ will be chosen later on. Using (\ref{ogranicenostUN}) we estimate $I_1$ by
\begin{multline}
I_1=(h^{|\alpha|^{\s}}|\alpha|^{\t |\alpha|^{\s}})^{-1}\Big|\int_{|\xi|\leq 1}\xi ^{\alpha}
\widehat u_N(\xi)e^{2\pi i x\xi}d\xi\Big| \\[1ex]
\leq  C(h^{|\alpha|^{\s}}|\alpha|^{\t |\alpha|^{\s}})^{-1}\int_{|\xi|\leq 1}\langle\xi\rangle^M d \xi.\nonumber
\end{multline}
If $h\geq 1$ we conclude that $I_1\leq C_1$ where $C_1$ does not depend on $\alpha$.
To estimate $I_2$, note that by (\ref{uslov4}) we have
\begin{eqnarray*}
I_2&=&(h^{|\alpha|^{\s}}|\alpha|^{\t |\alpha|^{\s}})^{-1}
\Big|\int_{|\xi|>1}\xi ^{\alpha}\widehat u_N(\xi)e^{2\pi i x\xi}d\xi\Big|  \\
&\leq& A (h^{|\alpha|^{\s}}|\alpha|^{\t |\alpha|^{\s}})^{-1} k^{N^{\s}}N^{\t N^{\s}}
\int_{|\xi|>1}|\xi|^{|\alpha|-N}d\xi \leq C (k^{2^{\s-1}}/h)^{|\alpha|^{\s}},
\end{eqnarray*}
where for the last inequality we chose $N=|\alpha|+d+1$, and use $\widetilde{(M.2)'}$ property of
$M_p^{\t,\s}$, $p\in \N$.
Now, for $h>k^{2^{\s-1}}$ we conclude that $I_2\leq C_2$, and $C_2$ does not depend on $\alpha$.
Hence, if we take $h>\max\{1,k^{2^{\s-1}}\}$, we conclude that
$u\in \E_{\{\t,\s\}}(\Omega)$, and the statement is proved.
\end{proof}

\begin{prop}
\label{potrebanUslov}
Let $\Omega\subseteq K\subset\subset U$, $\overline{\Omega} \subset K$,
$u\in \D'(U)$, and let $\{\chi_N\}_{N\in \N}$ be the ${\tilde\t},\s$-admissible sequence with respect to $K$, where ${\tilde\t}= \t^{\s/(\s -1)}$, $\t>0 $, $\s>1$ .
If $u\in \E_{\{\t,\s\}}(\Omega)$,
then $\{\chi_N u\}_{N\in \N}$ is bounded in $\E'(U)$,
${\chi}_N u=u$ on $\Omega$, and
\begin{equation}
\label{uslov2ro}
|\widehat{\chi_N u}(\xi)|\leq
A \frac{h^{N} N!^{{\tilde\t}^{-1/\s}/{\s}} }{|\xi|^{\lfloor ({N/{\tilde\t}})^{1/\s} \rfloor}}, \quad N\in { \N},\,\xi\in \Rd\backslash\{0\}.
\end{equation} That is, after enumeration $N\to {\tilde \t}N$, $\{\chi_N u\}_{N\in \N}$ satisfies \eqref{uslov3'}.
for some $A,h>0$.
\end{prop}

\begin{proof}
Put $u_N=\chi_N u$, $N\in \N$. By the Remark \ref{PaleyRemark},
$\{u_N\}_{N\in \N}$ is bounded in $\E'(U)$. Note also that $u_N=u$ on $\Omega$ and $\supp u_N\subseteq K$.

Since $u\in \E_{\{\t,\s\}}(\Omega)$, from (\ref{ocenaNiz}) for
$|\alpha|\leq \lf (N/{\tilde \t})^{1/\s} \rf$, $x\in \Omega$, and for some $k>1$
we obtain
\begin{eqnarray}
\label{u_Nkonstrukcija1}
|D^{\alpha}u_N(x)| &\leq& \sum_{\beta\leq\alpha}{\alpha \choose \beta}|D^{\alpha-\beta}\chi_N(x)||D^{\beta}u(x)|\nonumber\\
&\leq& || u ||_{\E_{\t,\s,k}(\Omega)} \sum_{\beta\leq\alpha}{\alpha \choose \beta} A^{|\alpha-\beta|+1} {\lf N^{1/\s}\rf}^{|\alpha-\beta|}k^{|\beta|^{\s}}|\beta|^{\t |\beta|^{\s}}\nonumber\\
&\leq& A || u ||_{\E_{\t,\s,k}(\Omega)} (2A)^{\lf (N/{\tilde \t})^{1/\s} \rf}\lf N^{1/\s} \rf^{\lf (N/{\tilde \t})^{1/\s} \rf}k^{N/{\tilde \t}}\lf N^{1/\s} \rf^{\frac{\t N}{{\tilde \t}}}\nonumber\\
&\leq& A || u ||_{\E_{\t,\s,k}(\Omega)} B^{N} N^{{\frac{1}{\s} (\frac{1}{\t})^{1/(\s-1)}}N^{1/\s}}N^{{\frac{1}{\s} (\frac{1}{\t})^{1/(\s-1)}}N}
\end{eqnarray} for some $B>0$, where for the last inequality we have used that $ \tilde \t = \t^{\s/(\s -1)}$.

Next we note that there exists $c>0$, such that
$$
N^{1/\s}\ln N\leq c N^{1/\s}N^{1-1/\s}=c N,
$$
wherefrom $N^{{\frac{1}{\s} (\frac{1}{{\t}})^{1/(\s-1)}}N^{1/\s}}\leq C^{N}$
for some $C>1$ (which depends on $\t$ and $\s$). Hence (\ref{u_Nkonstrukcija1}) can be estimated by
\begin{equation}
\label{u_Nkonstrukcija2}
|D^{\alpha}u_N(x)|\leq   A || u ||_{\E_{\t,\s,k}(\Omega)} h^{N} N^{{\frac{1}{\s} (\frac{1}{{\t}})^{1/(\s-1)}}N},
\end{equation}
for some $h>0$.
Applying the Fourier transform to (\ref{u_Nkonstrukcija2}) for $|\alpha|=\lf (N/{\tilde\tau})^{1/\s} \rf$ we obtain
\begin{equation}
\label{u_nKonstrukcija3}
|\widehat u_N(\xi)|\leq   A || u ||_{\E_{\t,\s,k}(\Omega)}  \frac{h^{N} N^{{\frac{1}{\s} (\frac{1}{\t})^{1/(\s-1)}}N}}{|\xi|^{\lfloor (N/{\tilde \t})^{1/\s} \rfloor}}, \quad N\in { \N},\,\xi\in \Rd\backslash\{0\}.
\end{equation}
Finally, after the enumeration $N\to {\tilde\tau} N$, we note that (\ref{u_nKonstrukcija3})
and Stirling's formula imply (\ref{uslov2ro}), and  the proposition is proved.
\end{proof}

\begin{rem} \label{RemarkSlucajS=1}
Sequence of functions $\{\varphi_N\}_{N\in \N}$
"analytic up to the order $N$" is
used to extended results from  \cite{HermanderKnjiga}
to Gevrey type ultradistributions, cf. \cite[Proposition 1.4.10, Corollary 1.4.11]{Rodino}.
When $\t>0$,  $\s>1$ and $\beta=0$ in (\ref{ocenaNiz}) we obtain
$$
\sup_{x\in K}|\partial^{\alpha} \chi_N (x)|
\leq A^{|\alpha|+1}\frac{\lfloor N^{1/\s}\rfloor ^{|\alpha|}}{|\alpha|^{\frac{1}{\s}|\alpha|}}|\alpha|^{\frac{1}{\s}|\alpha|}
$$
\begin{equation}
\label{potrbanuslovOcena}
\leq {A^{|\alpha|+1}}\sup_{r>0}\frac{N^{r/\s}}{r^{r/\s}}|\alpha|^{\frac{1}{\s}|\alpha|}\nonumber
= A^{|\alpha|+1} e^{\frac{1}{e\s}N} |\alpha|^{\frac{1}{\s}|\alpha|}, \quad |\alpha|\leq
\lfloor (N/\t)^{1/\s}  \rfloor,
\end{equation}
so  $\chi_N$ might be called
"quasi-analytic up to the order $\lfloor (N/\t)^{1/\s}  \rfloor$".
When $\s\to 1^+$
the order of quasi-analyticity of $\chi_N$ tends to infinity (for fixed $N\in \N$)
for $0<\t<1$, while for $\t>1$ it tends to zero.
Therefore the study of the "critical" behavior when
$\s\to 1^{+}$ is possible only if $\t$ depends on $\s$.

In particular, when $\s=1$ and $\t\not=1$, the proof of Proposition \ref{potrebanUslov} fails, while
for $\t=\s=1$ Proposition \ref{potrebanUslov} coincides with necessity part of \cite[Proposition 8.4.2.]{HermanderKnjiga}.
\end{rem}

\section{Wave-front sets with respect to $\E_{\t,\s}$} \label{sec-3}

In this section we define wave front sets which detect singularities that are
"stronger" then classical $C^{\infty}$ singularities and "weaker" then Gevrey type singularities,
which is done within the framework of the regularity classes
${\E}_{\tau, \s}(U)$.

\begin{de}
\label{Wf_t_s}
Let $\t>0$ and $\s>1$, $u \in \D'(U)$, and $(x_0,\xi_0)\in U\times\Rd\backslash\{0\}$.
Then $(x_0,\xi_0)\not \in {\WF}_{\{\t,\s\}}(u)$ (resp. ${\WF}_{(\t,\s)}(u)$) if there exists
open neighborhood $\Omega \subset U$ of $x_0$,
a conic neighborhood $\Gamma$ of $\xi_0$,
and a bounded sequence
$\{u_N\}_{N\in \N}$ in $\E'(U)$ such that $u_N=u$ on $\Omega$ and
\eqref{uslov3'} holds
for some constants $A,h>0$ (resp. for every $h>0$ there exists $A>0$).
\end{de}

Note that ${\WF}_{\{\t,\s\}}(u)$, $u\in \D'(U)$,
is a closed subset of $U\times\Rd\backslash\{0\}$ and for $\t>0$ and $\s>1$ we have
$$
{\WF}_{\{\t,\s\}}(u)\subseteq {\WF}_{\s}(u) \subseteq {\WF}_{\{1,1\}}(u)={\WF}_A (u),
$$
where $ {\WF}_{\s}(u)  $ is the Gevrey wave-front set.
If $ 0<\t<1$ and $\s = 1$ then $ {\WF}_A (u) \subseteq {\WF}_{\{\t,1\}}(u)$.

\par

Next we prove an important fact on microlocal regularity.

\begin{lema}
\label{Singsuplema}
Let $\t>0$, $\s>1$,  $u\in\D'(U)$, $K\subset\subset U$,
and let $\{\chi_N\}_{N\in \N}$ be a ${\tilde \t}, \s$-admissible sequence with respect to $K$ with ${\tilde\t}= \t^{\s/(\s -1)}$.
Then $\{\chi_N u\}_{N\in \N}$ is a bounded sequence in $\E'(U)$,
and if ${\WF}_{\{\t,\s\}}(u) \cap (K\times F)=\emptyset$,
where $F$ is a closed cone, then there exist $A,h>0$ such that
\be
\label{SingsuplemaUslov}
|\widehat{\chi_N u}(\xi)|\leq
A \frac{h^{N} N!^{{\tilde\t}^{-1/\s}/{\s}} }{|\xi|^{\lfloor ({N/{\tilde\t}})^{1/\s} \rfloor}},
\quad N\in {\N}\,,\xi\in F\,.
\ee
\end{lema}

\begin{proof}
Let $(x_0,\xi_0)\in K\times F$, and set $r_0:=r_{{x_0,\xi_0}}>0$.
Furthermore, let $\{\chi_N \}_{N\in \N}$ be the ${\tilde \t}, \s$-admissible sequence with respect to
$\overline{B_{r_0}(x_0)}$, $\overline{B_{r_0}(x_0)}\subseteq \Omega\subseteq K$.
Boundedness of $\{\chi_Nu \}_{N\in \N}$ follows by Remark \ref{PaleyRemark}.

Since $(x_0,\xi_0)\not \in {\WF}_{\{\t,\s\}}(u)$
we choose $u_N$, $\Omega$ and $\Gamma$ as in Definition \ref{Wf_t_s} so that
\begin{equation}
\label{uslov1-sigma}
|\widehat u_N(\xi)|\leq A\, \frac{h^{N} N!^{\t/\s} }{|\xi|^{\lfloor N^{1/\s} \rfloor}},
\quad N\in { \N},\,\xi\in \Gamma,
\end{equation}
for some $A,h>0$. Recall,  the condition (\ref{uslov1-sigma})  is equivalent to
\begin{equation}
\label{uslov2-sigma}
|\widehat u_N(\xi)|\leq A\, \frac{h^{N} N!^{{\tilde \t}^{-1/\s}/\s} }{|\xi|^{\lfloor (N/{\tilde \t})^{1/\s} \rfloor}},
\quad N\in { \N},\,\xi\in \Gamma,
\end{equation}
after applying Stirling's formula and  enumeration $N\to N/{\tilde \t}$.

Let $\Gamma_0$ be an open conical neighborhood of $\xi_0$ with the closure contained in $\Gamma$
and choose $\varepsilon >0$ such that $\xi-\eta \in \Gamma$ when $\xi\in \Gamma_0$ and $|\eta|<\varepsilon |\xi|$.
Then, since $\chi_N u=\chi_N u_N$, we write
$$
\widehat {\chi_N u}(\xi)=\Big(\int_{|\eta|<\varepsilon |\xi|} +\int_{|\eta|\geq\varepsilon |\xi|}\Big)
\widehat \chi_N (\eta)\widehat u_N(\xi-\eta)\,d\eta =I_1+I_2\,,\quad \xi\in \Gamma_0, N\in \N\,.
$$
To estimate $I_1$, note that for $|\eta|<\varepsilon |\xi|$ we have
$$|\xi-\eta|\geq |\xi|-|\eta|> (1-\varepsilon)|\xi|.$$
Thus, by using (\ref{chiNPaley}) for $\alpha=0$ and $|\beta|=d+1$ and (\ref{uslov2-sigma}), we obtain
\begin{multline}
|I_1|=\Big| \int_{|\eta|<\varepsilon |\xi|} \widehat \chi_N (\eta)\widehat u_N(\xi-\eta)\,d\eta \Big| \\[1ex]
\leq \int_{|\eta|<\varepsilon |\xi|} |{\widehat \chi_N} (\eta)| A \frac{h^{N} N!^{{\tilde \t}^{-1/\s}/\s} }{|\xi-\eta|^{\lfloor (N/{\tilde \t})^{1/\s} \rfloor}}d\eta
\\[1ex]
\leq A \frac{h^{N}N!^{{\tilde \t}^{-1/\s}/\s} }{((1-\varepsilon)|\xi|)^{\lfloor (N/{\tilde \t})^{1/\s} \rfloor}}\int_{{\bf R}^d}\langle \eta \rangle^{-d-1} d\eta
\leq A_1 \frac{h_1^{N}N!^{{\tilde \t}^{-1/\s}/\s}}{|\xi|^{\lfloor (N/{\tilde \t})^{1/\s} \rfloor}},\quad \xi\in \Gamma_0, N\in \N.\nonumber
\end{multline}
To estimate $I_2$, note that for $|\eta|\geq \varepsilon |\xi|$ we have
$$
|\xi-\eta|\leq |\xi|+|\eta|\leq (1+1/\varepsilon)|\eta|,
$$
and thus, using (\ref{chiNPaley}) for $|\alpha|=\lfloor (N/{\tilde \t})^{1/\s} \rfloor$, together with (\ref{ogranicenostUN}) and
$$
{\lfloor N^{1/\s} \rfloor}^{\lfloor (N/\t)^{1/\s} \rfloor}
\leq{ N }^{{1/\s} (1/\t)^{1/\s} N }\leq C^{N}{ N! }^{\t^{-1/\s}/\s},
$$
for every $\beta \in {\mathbf N}^d$ and some $M>0$ we obtain
\begin{eqnarray*}
|I_2|&=&\Big| \int_{|\eta|\geq \varepsilon |\xi|} \widehat \chi_N (\eta)\widehat u_N(\xi-\eta)\,d\eta \Big|\\
&\leq&\frac{A_{\beta}^{\lfloor (N/{\tilde \t})^{1/\s} \rfloor+1} {\lfloor N^{1/\s} \rfloor}^{\lfloor (N/{\tilde \t})^{1/\s} \rfloor}}{(\varepsilon |\xi|)^{\lfloor (N/{\tilde \t})^{1/\s} \rfloor}}\int_{|\eta|\geq \varepsilon |\xi|} \langle\eta\rangle^{-|\beta|} C \langle\xi-\eta\rangle^{M} \,d\eta\\
&\leq& \frac{A_{\beta}^{N +1} {\lfloor N^{1/\s} \rfloor}^{\lfloor (N/{\tilde \t})^{1/\s} \rfloor}}{(\varepsilon |\xi|)^{\lfloor (N/{\tilde \t})^{1/\s} \rfloor}}
\int_{{\bf R}^d}\langle\eta\rangle^{-|\beta|} \langle(1+1/\varepsilon)\eta\rangle^{M}  ,d\eta \\
&\leq&  \frac{A'^{N+1} N!^{{\tilde \t}^{-1/\s}/\s} }{|\xi|^{\lfloor (N/{\tilde \t})^{1/\s} \rfloor}},\quad \xi\in \Gamma_0,
\end{eqnarray*}
for some $A'>0$, where we have chosen $|\beta|=M+d+1$.

Thus, the statement follows for $(x,\xi)\in B_{r_0}(x_0)\times \Gamma_0$.

In order to extend the result to $K\times F$ we use the same idea
as in the proof of \cite[Lemma 8.4.4]{HermanderKnjiga}.
Since the intersection of $F$ with the unit sphere is a compact set,
there exists a finite number of balls $B_{r_{{x_0,\xi_j}}}(x_0)$,
and cones $\Gamma_j$ that covers $F$, $j\leq n$, $n\in \Z$, and note that (\ref{SingsuplemaUslov}) remains valid if $\{\chi_N\}_{N\in \N}$ is chosen so that $\dss \supp\chi_N \subseteq B_{r_{x_0}}:=\bigcap_{j=1}^{n}B_{r_{x_0,\xi_j}}(x_0)$.

Moreover, since $K$ is  compact set, it is covered by a finite number of balls $B_{r_{x_k}}$, $k\leq n$, $n\in \Z$.
By \cite[Lemma 5.1.]{Komatsuultra1} there exist
non-negative functions $\chi_k\in C_0^{\infty}(B_{r_{x_k}/2})$, $k\leq n$, such that $\suml_{k=1}^n \chi_k=1$
on a neighborhood of $K$. Next, for every $N\in \N$ we choose a non-negative function
$\phi_N\in C_0^{\infty}(B_{r_{x_k}/2})$ such that
$\int \phi_N (x) =1$
and
$$
\sup_{x\in K}
|D^{\alpha}{\phi_{N}}(x)|\leq C^{|\alpha|} \lfloor N^{1/\s}\rfloor ^{|\alpha|},
$$
for $|\alpha|\leq \lfloor (N/\tilde\t)^{1/\s}\rfloor$, where the constant $C>0$ depends on $\t$ and $\s$,
cf. \cite[Theorem 1.4.2.]{HermanderKnjiga}.
Now, for $\chi_{N,k}=\phi_N * \chi_k$, we have
$ \suml_{k=1}^n\chi_{N,k}=1$ in a neighborhood of $K$, and each $\chi_{N,k}$, $1\leq k\leq n$,
satisfies (\ref{ocenaNiz}).

To conclude the proof we note that if $\{\chi_N\}_{N\in \N}$ is a ${\tilde \t}, \s$-admissible sequence
with respect to $K$, then $\chi_N \chi_{N,k} $ also satisfies estimate of type (\ref{ocenaNiz}), for $1\leq k\leq n$. This follows by simple application of Leibniz rule.
Thus, (\ref{SingsuplemaUslov}) holds if we replace $\chi_N$ by $\chi_N \chi_{N,k}$. Since $\suml_{k=1}^n\chi_N\chi_{N,k}=\chi_N$, the result follows.
\end{proof}

Next we give a short comment on $\WF_{(\t,\s)}(u)$, $u\in \D'(U)$.
From our analysis it follows that the regularity related to the complement of $\WF_{\{\t,\s\}}$
is described by the (microlocal) regularity of $\E_{\{\t,\s\}}$.
Therefore the following Corollary follows  from the embeddings given by \eqref{RoumieuBeurling}.

\begin{cor}
\label{PosledicaPresekUnija}
Let $u\in \D'(U)$, $t>1$. Then for $0<\t<\rho $ and $\s>1$ it holds
$$\WF(u)\subseteq \WF_{\{\rho,\s\}}(u)\subseteq \WF_{(\rho,\s)}(u)\subseteq \WF_{\{\t,\s\}}(u) \subseteq\bigcap_{t>1}\WF_t(u)\subseteq \WF_A(u)\,,$$ where $\WF_t$ and $\WF_A$ are Gevrey and analytic wave front sets, respectively.
\end{cor}

\subsection{Pseudolocal  property of $\WF_{\t,\s}$} \label{pseloc}

We refer to \cite{PTT-02} for a more general result, and prove here only
the pseudolocal property of the wave-front set $\WF_{\{\t,\s\}}(u),$ $u\in \D'(U)$.

\begin{te} \label{pseudolocal}
Let
$$ P(x,D)=\sum_{|\alpha|\leq m} a_{\alpha } (x) D^{\alpha}$$ be a differential operator of order $m$ on $U$
with $a_\alpha \in \E_{\{\t,\s\}}(U)$, $|\alpha|\leq m $, and
let $u\in \D'(U)$, $\t>0, \s>1$. Then
$$
\WF_{\{\t,\s\}}(P(x,D) u)\subseteq\WF_{\{\t,\s\}}(u),
$$
\end{te}

The statement directly follows from the next lemma.

\begin{lema}
\label{zatvorenostWfizvodi}
Let $u\in \D'(U)$, $\t>0, \s>1$. Then
$$
\WF_{\{\t,\s\}}(\partial_j u)\subseteq\WF_{\{\t,\s\}}(u), \;\;\; 1\leq j\leq d.
$$
If, in addition
$\dss \phi \in \E_{\{\t,\s\}}(U)$, then
\be \label{closedmultipl}
\WF_{\{\tau,\s\}}(\phi u)\subseteq \WF_{\{\tau,\s\}}(u).
\ee
\end{lema}

\begin{proof}
We refer to \cite[Lemma 4.1]{PTT-02} for the first part and prove here only \eqref{closedmultipl}.

Set $\tilde \t=\t^{\frac{\s}{\s-1}}$ and fix $(x_0,\xi_0)\not\in WF_{\{\t,\s\}}(u)$. Then by the definition, there exists open conic neighborhood $\Omega\times \Gamma$ of $(x_0,\xi_0)$ and a bounded sequence $\{u_N\}_{N \in \N}$ in $\dss \E'(U)$ such that
$u_N=u$ on $\Omega$ and

\begin{equation}
\label{MnozenjesaFi1}
|\widehat u_N(\xi)|\leq A\, \frac{h^{N} N!^{\t/\s} }{|\xi|^{\lfloor N^{1/\s} \rfloor}}, \quad N\in { \N},\,\xi\in \Gamma.
\end{equation}
Choose a compact neighborhood $K_{x_0}\subset\subset \Omega$ of $x_0$, and let
$\{\chi_N\}_{N \in N}$ be $\tilde \t, \s$-admissible sequence with respect $K_{x_0}$.
Set $\widetilde{\chi}_N=\phi \chi_N$, $N\in \N$, and note that $\widetilde{\chi}_N u=\widetilde{\chi}_N u_N$. Since $M_p^{\t,\s}=p^{\t p^{\s}}$ satisfies $\widetilde{(M.2)'}$ (see Lemma \ref{osobineM_p_s}) for some positive increasing sequence $C_q$, $q\in \N$, and $h>1$ we obtain

\begin{multline}
\label{MozenjeSaFi}
|D^{\alpha+\beta}\widetilde{\chi}_N(x)| \leq
\sum_{\delta\leq \alpha}\sum_{\gamma\leq \beta} {\alpha\choose \delta}{\beta\choose \gamma} |D^{\alpha-\delta+\beta-\gamma}\widetilde{\chi}_N(x)| |D^{\gamma+\delta}\phi(x)|
\\[1ex]
\leq \sum_{\delta\leq \alpha}\sum_{\gamma\leq \beta} {\alpha\choose \delta}{\beta\choose \gamma}A_{\beta}^{|\alpha-\delta|+1} \lf N^{1/\s}\rf^{|\alpha-\delta|}h^{|\gamma+\delta|^{\s}+1}|\gamma+\delta|^{\t|\gamma+\delta|^{\s}}
\\[1ex]
\leq (2h')^{|\beta|^{\s}+1} \sum_{\delta\leq \alpha} {\alpha\choose \delta}A_{\beta}^{|\alpha-\delta|+1} \lf N^{1/\s}\rf^{|\alpha-\delta|}(C_{\beta}h')^{|\delta|^{\s}}|\delta|^{\t |\delta|^{\s}},
\\[1ex]
\end{multline} for $x\in K_{x_0}$, $|\alpha|\leq \lf(N/\tilde\t)^{1/\s}\rf$, $\beta\in {\bf N}^d$, where $h'=h^{2^{\s-1}}$. Now it is clear that by putting $|\alpha|=0$ in \eqref{MozenjeSaFi} we obtain,
$$\dss |D^{\beta}\widetilde{\chi}_N(x)|\leq C'_{\beta},\quad x\in K_{x_0},$$ and hence by applying Fourier transform it follows
$$
\dss|\widehat{\widetilde{\chi}}_N(\xi)|\leq C'_{\beta}\langle\xi\rangle^{-|\beta|},\quad \beta\in {\mathbf N}^d, \xi\in \Gamma,
$$ for suitable $C'_{\beta}>0$.
In particular, since $\dss \E_{\{\t,\s\}}(U)\hookrightarrow C^{\infty}(U)$ it follows that $\widetilde{\chi}_N=\phi {\chi}_N$, $N\in \N$, is bounded in $C^{\infty}(U)$ and hence $\widetilde{\chi}_N u$, $N\in \N$, is bounded in $\E'(U)$.

Moreover, note that by the same type of estimates as in \eqref{u_Nkonstrukcija1} for $|\alpha|=\lf(N/\tilde\t)^{1/\s}\rf$ and $\beta\in {\mathbf N}^d$, by \eqref{MozenjeSaFi} we obtain that

$$
|D^{\alpha+\beta}\widetilde{\chi}_N(x)|\leq  C_{\beta}^{''N+1} N^{{\frac{1}{\s}(\frac{1}{\tilde\t})^{1/\s}}N}, \beta\in {\mathbf N}^d, x\in K_{x_0}
$$ and hence after applying Fourier transform it follows
\be
\label{MnozenjeSafi2}
|\widehat{\widetilde{\chi}}_N(\xi)|\leq  C_{\beta}^{''N+1} N^{{\frac{1}{\s}(\frac{1}{\tilde\t})^{1/\s}}N}\langle\xi \rangle^{-|\alpha|-|\beta|},\quad \beta\in {\mathbf N}^d, \xi\in \Gamma,
\ee for some constants $C''_{\beta}>0$.

Now using (\ref{MnozenjesaFi1}) and (\ref{MnozenjeSafi2}) and arguing in the same way as in the proof of Lemma \ref{Singsuplema}, one can find open cone $\Gamma_0 \subseteq \Gamma$ such that
$$
|\widehat{\widetilde\chi_N u}(\xi)|\leq A\, \frac{h^{N} N^{\frac{{\tilde\t}^{-1/\s}}{\s} N} }{|\xi|^{\lfloor (N/\tilde\t)^{1/\s} \rfloor}}, \quad N\in { \N},\,\xi\in \Gamma_0,
$$ for suitable $A,h>0$. After enumeration $N\to \tilde \t N$ the statement follows.
\end{proof}

\subsection{Intersections and unions of  $\WF_{\t,\s}$ and the corresponding singular supports}

It turns out that the regularity related to the complement of the unions and intersections
of wave-front sets $\WF_{\t,\s}$, $\t>0$, $\s>1$,
coincides with the regularity given by (\ref{BorderLinePresek})-(\ref{BorderLineUnija}).

In particular, for $u\in \D'(U)$, we consider
\be
\label{WF01}
\WF_{0,1}(u)=\bigcap_{\s>1}\bigcap_{\t>0}\WF_{\t,\s}(u),
\ee
\be
\label{WFinfty1}
\WF_{\infty,1}(u)=\bigcap_{\s>1}\bigcup_{\t>0}\WF_{\t,\s}(u),
\ee
\be
\label{WF0infty}
\WF_{0,\infty}(u)=\bigcup_{\s>1}\bigcap_{\t>0}\WF_{\t,\s}(u),
\ee
\be
\label{WFinftyinfty}
\WF_{\infty,\infty}(u)=\bigcup_{\s>1}\bigcup_{\t>0}\WF_{\t,\s}(u).
\ee
where $\WF_{\t,\s}(u)$ denotes either $\WF_{\{\t,\s\}}(u)$ or $\WF_{(\t,\s)}(u)$.

From Corollary \ref{PosledicaPresekUnija} we have
$$
\dss \bigcap_{\t>0}\WF_{\{\t,\s\}}(u)=\bigcap_{\t>0}\WF_{(\t,\s)}(u)\;\;\;
\text{ and } \;\;\;
\bigcup_{\t>0}\WF_{\{\t,\s\}}(u)=\bigcup_{\t>0}\WF_{(\t,\s)}(u),
$$
so it is sufficient to observe $\WF_{\{\t,\s\}}(u)$ in \eqref{WF01} -- \eqref{WFinftyinfty}.

By \cite[Lemma 3.4]{PTT-02}, i.e.
$$\bigcup_{\t>0}\WF_{\t,\s_2}(u)\subseteq \bigcap_{\t>0}\WF_{\t,\s_1}(u),\;\;\; u\in \D'(U), \;\s_2>\s_1\geq 1,
$$
we have the following:
\begin{multline}
\WF(u) \subseteq\WF_{0,1}(u)\subseteq \WF_{\infty,1}(u)\\
\subseteq \WF_{0,\infty}(u)\subseteq \WF_{\infty,\infty}(u)\subseteq \bigcap_{\t>1}\WF_{\t}(u)\,,\nonumber
\end{multline} where $\WF(u)$ and $\WF_{\t}(u)$ are the classical and the Gevrey wavefront sets, respectively, see also
\cite[Corollary 3.1]{PTT-02}.

\par

Next we define singular support of distributions with respect to classes $\E_{\{\t,\s\}}$,$\t>0$ and $\s>1$,
and the corresponding borderline cases $\t \in \{ 0, \infty \} $ and $ \s \in \{ 1, \infty \} $ defined by
\eqref{BorderLinePresek} -- \eqref{BorderLineUnija}.

\begin{de}
\label{DefinicijaSingSup}
Let $\t \in[0,\infty]$ and $\s\in[1, \infty]$, $u\in \D'(U)$ and $x_0\in U$.
Then $x_0\not \in \sing_{\{\tau,\s\}}(u)$ if and only if there exists neighborhood $\Omega$ of $x_0$ such that $u\in \E_{\{\t,\s\}}(\Omega)$.
\end{de}

Let $\pi_1:U\times\Rd\backslash \{0\}\to U$ denotes the standard projection given by $\pi_1(x,\xi)=x$.
From Propositions \ref{dovoljanUslov}, \ref{potrebanUslov}, and Lemma \ref{Singsuplema} it follows that for a given $u\in \D'(U)$,
$\t>0$ and $\s>1$,  we have
$ \sing_{\{\t,\s\}}(u)=\pi_1(\WF_{\{\t,\s\}}(u))$.

For the borderline cases $\t \in \{ 0, \infty \} $ and $ \s \in \{ 1, \infty \} $  we have the following.

\begin{te} \label{projections}
Let there be given $u\in \D'(U)$ and let
$\pi_1:U\times\Rd\backslash \{0\}\to U$ be the standard projection. Then
$$
\pi_1(\WF_{\infty,\infty}(u))=\sing_{0,1}(u)\,,
$$
$$
\pi_1(\WF_{0,1}(u))=\sing_{\infty,\infty}(u)\,,
$$
$$
\pi_1(\WF_{\infty,1}(u))=\sing_{0,\infty}(u)\,,
$$
$$
\pi_1(\WF_{0,\infty}(u))=\sing_{\infty,1}(u).
$$
\end{te}

\begin{proof}
We prove here only $ \pi_1(\WF_{0,\infty}(u))=\sing_{\infty,1}(u) $
and leave the other equalities to the reader.

Assume that $x_0 \not \in \pi_1(\WF_{0^+,\infty}(u))$, so that there is a compact neighborhood $K \subset \subset U$ of $x_0$
such that
\be
\label{KomplementWF}
K\times \Rd\backslash\{0\}\subseteq(\WF_{0^+,\infty}(u))^c=\bigcap_{\s>1}\bigcup_{\t>0}(\WF_{\{\t,\s\}}(u))^c,
\ee
where $(\WF_{\{\t,\s\}}(u))^c$ denotes the complement of the set $\WF_{\{\t,\s\}}(u)$ in $U\times \Rd\backslash\{0\}$.
Therefore, if $(x,\xi)\in K\times \Rd\backslash\{0\}$ then for every $\s>1$ there exist $\t_0>0$ such that $(x,\xi)\not \in \WF_{\{\t_0,\s\}}(u)$.

Let $\s>1$ be arbitrary but fixed, and set $\tilde \t_0=\t_0^{\s/(\s-1)}$.
From Lemma \ref{Singsuplema}, it follows that there is a $\tilde \t_0,\s$-admissible sequence
$\{\chi_N\}_{N\in \N}$ such that $u_N=\chi_N u$, $N\in \N$ is a  bounded sequence in $\E'(U)$, $u_N=u$ on some
$\Omega \subseteq K$, and
$$
|\widehat{\chi_N u}(\xi)|\leq A \frac{h^{N} N!^{{\tilde \t_0}^{-1/\s}/{\s}} }{|\xi|^{\lfloor ({N/{\tilde\t_0}})^{1/\s} \rfloor}},\quad N\in {\N}\,,\xi\in \Rd\backslash\{0\}\,,
$$
which after enumeration $N\to \tilde \t_0 N $ becomes
\be
\label{OcenaPresekUnijaWF1}
|\widehat{\chi_N u}(\xi)|\leq A \frac{h^{N} N!^{\t_0/{\s}} }{|\xi|^{\lfloor {N}^{1/\s} \rfloor}},\quad N\in {\N}\,,\xi\in \Rd\backslash\{0\}\,.
\ee
By Proposition \ref{dovoljanUslov} it follows that $u\in \E_{\{\t_0,\s\}}(U)$, and since $\s$ can be chosen arbitrary,
we conclude that $u\in \E_{\infty,1}(U)$ (see Proposition \ref{detectposition}).
Therefore   $ \sing_{\infty,1}(u) \subset \pi_1(\WF_{0,\infty}(u))$.

For the opposite  inclusion, assume that $x_0\not \in \sing_{\infty, 1}(u)$. Then  $u\in \E_{\infty,1}(\Omega)$, for some
$\Omega$ which is a neighborhood of $x_0$.
In particular, for every $\s>1$ there exists $\t_0>0$ such that $u\in\E_{\t_0,\s}(\Omega)$. Fix $\s>1$ and put $\tilde \t =\t_0^{\s/(\s-1)}$.
Now we use a $\tilde \t_0,\s$-admissible sequence $\{\chi_N\}_{N\in \N}$ and
Proposition \ref{potrebanUslov} implies (\ref{OcenaPresekUnijaWF1}). It follows that $(x_0,\xi)\in  (\WF_{\{\t_0,\s\}}(u))^c$
for every  $\s>1$ and for some $\t_0>0$. Hence, by the equality in (\ref{KomplementWF}) it follows that $(x_0,\xi)\not \in \WF_{0,\infty}(u)$
for every $\xi\in \Rd\backslash\{0\}$ and therefore $x_0\not \in \pi_1(\WF_{0,\infty}(u))$, wherefrom
$
\pi_1(\WF_{0,\infty}(u)) \subset \sing_{\infty,1}(u),
$
which finishes the proof.
\end{proof}


\subsection*{Acknowledgment}
This research is supported by Ministry of Education, Science and
Technological Development of Serbia through the Project no. 174024.
\par

\end{document}